\documentclass{amsart}
\usepackage{amssymb,setspace,tikz,xcolor,mathrsfs,listings,multicol}
\usepackage{rotating}
\usepackage[vcentermath,enableskew]{youngtab}
\usepackage{ytableau}
\usepackage[all,cmtip]{xy}
\usepackage{fullpage}

\usepackage[colorlinks=true, pdfstartview=FitV, linkcolor=blue, citecolor=blue, urlcolor=blue]{hyperref}

\newcommand{\column}[1]{\begin{bmatrix} #1 \end{bmatrix}}
\newcommand{\g}{\mathfrak{g}}
\newcommand{\asl}{\widehat{\mathfrak{sl}}} 
\newcommand{\wt}{\operatorname{wt}} 
\newcommand{\RC}{\operatorname{RC}} 
\newcommand{\trop}{\operatorname{trop}} 
\newcommand{\pr}{\operatorname{pr}} 
\newcommand{\ls}{\operatorname{ls}} 
\newcommand{\iso}{\approx} 

\newcommand{\markbox}{\hspace{-2pt}\times\hspace{-2pt}} 

\newcommand{\bfX}{\mathbf{X}} 
\newcommand{\bfx}{\mathbf{x}} 
\newcommand{\mcT}{\mathcal{T}}

\newcommand{\ZZ}{\mathbb{Z}}
\newcommand{\QQ}{\mathbb{Q}}
\newcommand{\RR}{\mathbb{R}}

\numberwithin{equation}{section}

\theoremstyle{plain}
\newtheorem{thm}{Theorem}[section]
\newtheorem{lemma}[thm]{Lemma}
\newtheorem{prop}[thm]{Proposition}
\newtheorem{conj}[thm]{Conjecture}

\theoremstyle{definition}

\newtheorem{remark}[thm]{Remark}
\newtheorem{example}[thm]{Example}

\usepackage{listings}
\lstdefinelanguage{Sage}[]{Python}
{morekeywords={False,sage,True},sensitive=true}
\definecolor{dblackcolor}{rgb}{0.0,0.0,0.0}
\definecolor{dbluecolor}{rgb}{0.01,0.02,0.7}
\definecolor{dgreencolor}{rgb}{0.2,0.4,0.0}
\definecolor{dgraycolor}{rgb}{0.30,0.3,0.30}

\definecolor{darkred}{rgb}{0.7,0,0} 
\newcommand{\defn}[1]{{\color{darkred}\emph{#1}}} 

\usepackage[colorinlistoftodos]{todonotes}

\title{Rigged configurations as tropicalizations of loop Schur functions}

\author{Travis Scrimshaw}
\address{School of Mathematics, University of Minnesota, 206 Church St. SE, Minneapolis, MN 55455}
\email{tscrimsh@umn.edu}
\urladdr{http://www.math.umn.edu/~tscrimsh/}

\begin{document}


\thanks{TS was partially supported by the National Science Foundation RTG grant NSF/DMS-1148634.}

\begin{abstract}
We conjecture an explicit formula for the image of a tensor product of Kirillov--Reshetikhin crystals $\bigotimes_{i=1}^m B^{1, s_i}$ under the Kirillov--Schilling--Shimozono bijection.
Our conjectured formula is piecewise-linear, where the shapes are given by the tropicalization of cylindric loop Schur functions and the riggings are given by the tropicalization of loop Schur functions. We prove that our formula changes the riggings by the correct amount based upon the time evolution of the corresponding box-ball system. We show that our formula is correct under the column splitting portion of the Kirillov--Schilling--Shimozono bijection and for $B^{1, s}$.
\end{abstract}

\maketitle

\section{Introduction}

The Korteweg-de Vries (KdV) equation~\cite{Boussinesq1877, KdV1895} is a non-linear partial differential equation that has been used to model shallow water waves in 1 dimension ({\it e.g.}, in a thin channel).
Kruskal and Zabusky~\cite{KZ64} noticed that the solutions separated into distinct solitary waves that retain their shape after interaction, which they called solitons.
In~\cite{GGKM74}, the technique of the inverse scattering transform was invented and applied to the KdV equation, showing that $m$-soliton solutions existed and the KdV equation is an exactly solvable model.

We will be focusing on both a discrete version of the KdV equation, Hirota's discrete KdV equation~\cite{Hirota81}, and an ultradiscrete version, the Takahashi and Satsuma box-ball system~\cite{TS90}.
Our methods to study the discrete KdV equation use the geometric crystals of Berenstein and Kazhdan~\cite{BK00, BK07} and the box-ball system use Kashiwara's crystal bases~\cite{K90, K91}.
More specifically, the discrete KdV equation solitons is modeled using the type $A_n^{(1)}$ affine geometric crystals from~\cite{KNO08, LP12}, and the box-ball system is realized in terms of type $A_n^{(1)}$ Kirillov--Reshetikhin (KR) crystals~\cite{HKOTY99, KKMMNN92}.

A crystal basis is a special basis of a representation of the quantum group $U_q(g)$ that behaves nicely in the limit $q \to 0$. Kashiwara showed~\cite{K91} that all highest weight representations admit a crystal basis.
However, for $U_q'(\g) = U_q([\g, \g])$ with $\g$ of affine type, there exists finite-dimensional representations that do not admit a crystal basis.
Yet, if we restrict to the class of Kirillov--Reshetikhin (KR) modules, then it is conjectured~\cite{HKOTY99, HKOTT02} that they all admit a crystal basis.
This was proven for type $A_n^{(1)}$ in~\cite{KKMMNN92}, all non-exceptional types in~\cite{FOS09}, and some special cases in the exceptional types~\cite{KMOY07, Yamane98}.

Now we restrict ourselves to when $\g = \asl_n$.
Kirillov--Reshetikhin (KR) crystals, the $U_q'(\asl_n)$-crystals corresponding to KR modules, are known to have many remarkable properties.
It is known that KR crystals are perfect crystals~\cite{FOS10}; in particular, any tensor product of KR crystals is connected.
Thus, there exists a unique $U_q'(\asl_n)$-crystal isomorphism, called the combinatorial $R$-matrix, that interchanges two factors.
We can express a box-ball system as a semi-infinite tensor product of KR crystals, where a series of combinatorial $R$-matrices encodes the time evolution~\cite{HHIKTT01}.

In~\cite{B31}, Bethe invented what is now known as the Bethe ansatz to solve the 1-dimensional spin-1/2 Heisenberg spin chain.
Kirillov and Reshetikhin in~\cite{KR86} gave (then conjectural) branching rules to $U_q(\mathfrak{sl}_n)$, where they, along with Kerov, developed combinatorial objects known as rigged configurations~\cite{KKR86, KR86}.
Baxter introduced the corner transfer matrix to solve integrable 2D lattice models by using 1-dimensional lattice paths, which have a natural interpretation as elements in a tensor product of KR crystals~\cite{B89}.
The Hamiltonian of the Heisenberg spin chain commutes with the row-to-row transfer matrix of the 2D lattice model and can be diagonalized simultaneously.
This suggests a relationship between the one-point function of the 2D lattice model and solutions to the Bethe ansatz, which led to the $X = M$ conjecture of~\cite{HKOTY99, KKMMNN92}.
By careful analysis, Kerov, Kirillov, and Reshetikhin in~\cite{KKR86, KR86} gave a bijection between rigged configurations and $(B^{1,1})^{\otimes m}$ on classically highest weight elements.
This was then expanded by Kirillov, Schilling, and Shimozono for general $\bigotimes_{i=1}^m B^{r_i, s_i}$ in~\cite{KSS02}.
It was shown in~\cite{DS06, S06, SW10} that rigged configurations could be given a $U_q'(\asl_n)$-crystal structure such that $\Phi$ is a full $U_q'(\asl_n)$-crystal isomorphism.

The bijection between rigged configurations and KR crystals, which we call the KSS bijection (sometimes known as the KKR bijection), is described recursively, but despite this, it has many remarkable properties.
The one we will focus on is that the KSS bijection sends the combinatorial $R$-matrix to the identity map on rigged configurations.
This results in rigged configurations describing the action-angle variables of the box-ball system and that the rows of the first partition $\nu^{(1)}$ correspond to the solitons of the box-ball system when there is no interaction~\cite{KOSTY06, Takagi05}.
Yet because of its recursive definition, an explicit (closed) description of the KSS bijection is difficult to construct.

The theory of geometric crystals was initiated in~\cite{BK00, BK07} as a algebro-geometric analog of Kashiwara's crystals in order to provide a rational description for the relationships between the Lusztig datum of a crystal.
Berenstein and Kazhdan constructed a geometric crystal that corresponds to, under a process called tropicalization, to the crystal basis of highest weight modules and $U_q^-$ in finite types.
In~\cite{KNO10}, a geometric crystal corresponding to the coherent limit of perfect crystals $B^{1,s}$ given by~\cite{KKM94}, which was then modified in~\cite{LP12} for type $\asl_n$ that tropicalizes to $B^{1,s}$.
Yamada has also given a geometric description of the combinatorial $R$-matrix for $R \colon B^{1,s} \otimes B^{1,s'} \to B^{1,s'} \otimes B^{1,s}$ of type $A_n^{(1)}$~\cite{Yamada01}.
Lam and Pylyavskyy then examined the ring of geometric $R$-matrix invariants in~\cite{LP13}, where they developed the ring of loop symmetric functions.

In~\cite{LPS15}, a conjectural description of the partitions that arise under the image of the KSS bijection for tensor products of the form $\bigotimes_{i=1}^m B^{1,s_i}$ in type $A_n^{(1)}$.
They proved that their conjecture holds for $\nu^{(1)}$, thus giving an explicit description of the rigged configuration corresponding to a state of a box-ball system, not only when they are not interacting.
Their description was based on the tropicalization of a simple ratio between two cylindric loop Schur functions, which are loop symmetric functions.
Moreover, in~\cite{KSY07}, a description of $\Phi^{-1}$ was given in terms of a tropicalization of the $\tau$ function from the Kadomtsev--Petviashvili (KP) hierarchy (see, {\it e.g.},~\cite{JM83} for more information).
In~\cite{HHIKTT01}, it was also shown that time evolution is a tropicalization of the non-autonomous discrete KP equation.

In this paper, we continue the work of~\cite{LPS15} by conjecturing an explicit formula for the riggings under the KSS bijection.
Our conjectured formula is based upon loop Schur functions whose shape grows with time.
We show that our conjectural formula increases by the correct amount, in that the growth is essentially given by adding a \emph{cylindric} semistandard tableau.
Moreover, we show that our conjectural formula respects one of the two parts of the KSS bijection: the column splitting.
(Note that the column reduction never arises in our situation.)
Additionally, we prove our conjectural formula for $B^{1,s}$, {\it i.e.} when there is precisely one factor.
As further evidence of our conjecture, the inverse map was described using piecewise linear maps in~\cite{KSY07}.

We now give some consequences of our conjecture.
This gives further evidence to~\cite{LP12, LPS15} that there is a natural notion of geometric KR crystals.
This work also suggests that the geometric analogs of rigged configurations is played by (cylindric) loop Schur functions.
Furthermore, it suggests that there is a geometric version of the $X = M$ conjecture as a geometric analog of energy was given in~\cite[Thm.~6.9]{LPS15}.
There is further evidence for this since the tropicalization of Baxter's corner transfer matrix is given by the tropicalization of the $\tau$-function of~\cite{KSY07}, which is used to describe $\Phi^{-1}$.
Thus, our formula could give a geometric version of Kostka polynomials, and hence, give geometric versions of Hall--Littlewood polynomials.
Moreover, geometric versions of virtual crystals ({\it i.e.}, (geometric) crystals invariant under diagram foldings) was given in~\cite{KNO08, KNO10}, thus by using virtual crystals, we could extend this to explicit descriptions of the bijection in types $C_n^{(1)}$, $A_{2n}^{(2)}$, $A_{2n}^{(2)\dagger}$, and $D_{n+1}^{(2)}$.
Furthermore, our formula could be used to give a more explicit description of certain statistics on marginally large tableaux~\cite{HL08} to rigged configurations~\cite{SalisburyS15} using a modified form of the bijection~\cite{SalisburyS16}.

Our conjectural formula allows us to give a continuous version of the bijection, giving a positive answer to a conjecture poised in~\cite{Okado15} for the special case of $\bigotimes_{i=1}^N B^{1, s_i}$ in type $A_n^{(1)}$.
This has applications to the tropical periodic Toda lattice, where a similar such map was given in~\cite{Takagi14} for type $A_1^{(1)}$ to linearize the (integrable) system.
Additionally, we could then take $s_i \in \RR_{>0}$ (instead of $s_i \in \ZZ_{>0}$ to form the tensor products $\bigotimes_{i=1}^m B^{1,s_i}$, which would give an analog of the discrete KdV equation.
Then by taking a suitable limit as $m \to \infty$ (with the size of the extra vacuum states varying with it), our formula could potentially be used to (re)construct solutions to the KdV equation.

This paper is organized as follows. In Section~\ref{sec:background}, we fix our notation and give the necessary background. In Section~\ref{sec:conjectures}, we state our conjectural formula and prove our formulas change following time evolution on rigged configurations. In Section~\ref{sec:KSS_bijection}, we show that our conjectural formula agrees with the column splitting map.

\section{Background}
\label{sec:background}

In this section, we give the necessary background on crystals, box-ball systems, and rigged configurations. Our partitions are in English convention.

\subsection{Crystals}

Consider the affine Lie algebra $\asl_n$ with index set $I$, Cartan matrix $(a_{ij})_{i,j \in I}$, simple roots $(\alpha_i)_{i \in I}$, simple coroots $(h_i)_{i \in I}$, and fundamental weights $(\Lambda_i)_{i \in I}$.
Let $P = \operatorname{span}_{\ZZ} \{ \Lambda_i \mid i \in I \}$ be the weight lattice, and $Q = \operatorname{span}_{\ZZ} \{ \alpha_i \mid i \in I \}$ be the weight lattice.
Let $I_0 = I \setminus \{0\}$ be the index set of the corresponding classical Lie algebra $\mathfrak{sl}_n$.
Denote the non-degenerate pairing by $\langle\ ,\ \rangle$, and recall that $\langle h_i, \alpha_j \rangle = a_{ij}$ and $\langle h_i, \Lambda_j \rangle = \delta_{ij}$.

Let $U_q'(\asl_n) = U_q([\asl_n, \asl_n])$ denote the quantum group of the derived subalgebra of $\asl_n$. For convenience, we define the \defn{level-$0$ fundamental weights} $\varpi_i = \Lambda_i - \Lambda_0$ for all $i \in I_0$.

The \defn{Kirillov--Reshetikhin (KR) crystal} $B^{1,s}$ is a $U_q'(\asl_n)$-crystal that consists of semistandard tableaux in a $1 \times s$ rectangle, which we identify with the left-to-right reading word.
We will also identify an element
\[
\underbrace{1 \cdots 1}_{x_1} \, \underbrace{2 \cdots 2}_{x_2} \, \cdots \, \underbrace{n \cdots n}_{x_n} \in B^{1, s},
\]
where $x_1 + x_2 + \cdots + x_n = s$,
with the vector $(x_1, x_2, \dotsc, x_n)$.
This is called the \defn{coordinate representation} or vector representation of $B^{1,s}$.

\begin{figure}
\[
\begin{tikzpicture}[scale=1.75,baseline=-4]
\node (1) at (0,0) {$\young(1)$};
\node (2) at (1.5,0) {$\young(2)$};
\node (d) at (3.0,0) {$\cdots$};
\node (n-1) at (4.5,0) {$\boxed{n-1}$};
\node (n) at (6,0) {$\young(n)$};
\draw[->,red] (1) to node[below]{\tiny$1$} (2);
\draw[->,blue] (2) to node[below]{\tiny$2$} (d);
\draw[->,dgreencolor] (d) to node[below]{\tiny$n-1$} (n-1);
\draw[->,orange] (n-1) to node[below]{\tiny$n$} (n);
\draw[->] (n) .. controls (5, 1) and (1, 1) .. node[above]{\tiny$0$} (1);
\end{tikzpicture}
\]
\caption{The crystal $B^{1,1}$.}
\label{fig:example_B11}
\end{figure}
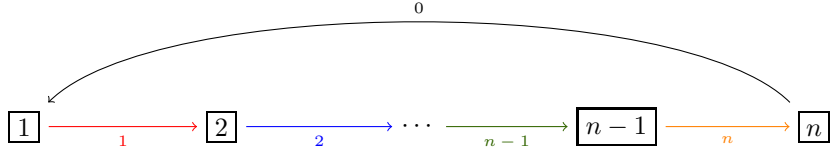

Following~\cite{KN94,Shimozono02}, we define the \defn{crystal operators} $e_i, f_i \colon B^{1,s} \to B^{1,s} \sqcup \{ 0 \}$ by
\begin{subequations}
\label{eq:ef_defn}
\begin{align}
\label{eq:e_defn}
e_i(x_1, \dotsc, x_i, x_{i+1}, \dotsc, x_n) & =
\begin{cases}
(x_1, \dotsc, x_i + 1, x_{i+1} - 1, \dotsc, x_n) & \text{if } x_{i+1} > 0, \\
0 & \text{otherwise},
\end{cases}
\\ \label{eq:F_defn}
f_i(x_1, \dotsc, x_i, x_{i+1}, \dotsc, x_n) & =
\begin{cases}
(x_1, \dotsc, x_i - 1, x_{i+1} + 1, \dotsc, x_n) & \text{if } x_i > 0, \\
0 & \text{otherwise},
\end{cases}
\end{align}
\end{subequations}
where $i \in I := \{0, 1, \dotsc, n-1\}$ and the indices are understood modulo $n$.
That is to say, if $x_n > 0$, then $f_0(x_1, \dotsc, x_n) = (x_1+1, \dotsc, x_n-1)$.
Furthermore, we define statistics
\begin{equation}
\label{eq:ep_phi_defn}
\varepsilon_i(b) = \max \{k \mid e_i^k(b) \neq 0\},
\qquad\qquad
\varphi_i(b) = \max \{k \mid f_i^k(b) \neq 0\},
\end{equation}
and define the weight function $\wt \colon B \to P$, where $P = \ZZ^n / (1, 1, \dotsc, 1)$, by
\begin{equation}
\label{eq:wt_defn}
\wt(x_1, \dotsc, x_n) = x_1 e_1 + \dotsm + x_n e_n.
\end{equation}
For the precise definition of a $U_q'(\asl_n)$-crystal, we refer the reader to~\cite{HK02}.

We will also require the more general KR crystals $B^{r,s}$.
We first note that the diagram automorphism given by $i \mapsto i+1$ descends to a twisted $U_q'(\asl_n)$-crystal isomorphism $\pr \colon B^{r,s} \to B^{r,s}$, which is known as \defn{promotion}~\cite{BST10, Haiman92, Shimozono02}.
Following~\cite{Shimozono02}, the KR crystal $B^{r,s}$ is the $U_q(\mathfrak{sl}_n)$-crystal $B^{r,s} \iso B(s\Lambda_r)$, which we realize as semistandard tableau with entries $\{1, 2, \dotsc, n\}$ of~\cite{KN94}, with the remaining crystal structure is described by
\[
e_0 = \pr \circ e_1 \circ \pr,
\hspace{80pt}
f_0 = \pr \circ f_1 \circ \pr,
\]
Equation~\eqref{eq:ep_phi_defn}, and Equation~\eqref{eq:wt_defn}, where $x_i$ counts the number of entries $i$ occurring in $b$.
We remark that on tableaux, $\pr$ is the promotion operator of Sch\"utzenberger~\cite{Sch72}.
We refer the reader to~\cite{OS08, FOS09} for more details on the KR crystals $B^{r,s}$.

Kashiwara~\cite{K91} has shown that $U_q'(\asl_n)$-crystals form a monoidal category under the following tensor product. Let $B_1$ and $B_2$ be KR crystals, then we form $B_1 \otimes B_2$ as the Cartesian product $B_1 \times B_2$ with the crystal structure
\begin{align*}
e_i(b_1 \otimes b_2) & = \begin{cases}
e_i(b_1) \otimes b_2 & \text{if } \varepsilon_i(b) > \varphi_i(b), \\
b_1 \otimes e_i(b_2) & \text{if } \varepsilon_i(b) \leq \varphi_i(b),
\end{cases}
\\ f_i(b_1 \otimes b_2) & = \begin{cases}
f_i(b_1) \otimes b_2 & \text{if } \varepsilon_i(b) \geq \varphi_i(b), \\
b_1 \otimes f_i(b_2) & \text{if } \varepsilon_i(b) < \varphi_i(b),
\end{cases}
\\ \wt(b_1 \otimes b_2) & = \wt(b_1) + \wt(b_2),
\end{align*}
with $\varepsilon_i$ and $\varphi_i$ given by Equation~\eqref{eq:ep_phi_defn}.

\begin{remark}
This is the opposite to the convention of Kashiwara.
\end{remark}

In~\cite{K02}, Kashiwara showed that $B = \bigotimes_{k=1}^m B^{r_k,s_k}$ is connected.
Moreover, there is a unique element $u_{\Lambda} \in B$ of weight $\Lambda = \sum_{k=1}^m s_k \varpi_{r_k}$.
Therefore, there is a unique crystal isomorphism $R \colon B \to B'$, where $b'$ is a reordering of the factors of $B$, given by $u_{\Lambda} \mapsto u'_{\Lambda}$, where $u' \in B'$ is the unique element of weight $\Lambda$, called the \defn{combinatorial $R$-matrix}.
Note that $R$ is generally not given by $a \otimes b \mapsto b \otimes a$. We draw the isomorphism $R \colon a \otimes b \to \widetilde{b} \otimes \widetilde{a}$ as
\[
\begin{tikzpicture}
\node (a) at (1, 0) {$a$};
\node (at) at (-1, 0) {$\widetilde{a}$};
\node (b) at (0, 1) {$b$};
\node (bt) at (0, -1) {$\widetilde{b}$};
\draw[->] (a) -- (at);
\draw[->] (b) -- (bt);
\end{tikzpicture}
\]
It is known that the category of KR crystals is a symmetric monoidal category from the natural action of the symmetric group via the combinatorial $R$-matrix.

\subsection{Box-ball systems}

Roughly speaking, a \defn{box-ball system} is a dynamical system where multicolored balls are placed in bins of various sizes.
A \defn{carrier} picks up an untouched $a$-colored ball from right-to-left and move it to the next available position.
This is done once for each color $a$.
A cluster of balls in adjacent bins is called a \defn{soliton}, and each soliton moves with speed equal to its size.
For more information on box-ball systems, we refer the reader to~\cite{IKT12}.

\begin{example}
\label{ex:box_ball_dynamics}
Consider a box-ball system with all bins have size $1$. The following is the evolution of a soliton of length $2$, $3$, and $4$:
\begin{align*}
   t = 0: & \cdots {\color{gray} 1 1 1 1 1 1 1 1 1 1 1 1 1 1 1 1 1 1 1 1 1 1 1 1 1 1 1 1} 2 4 {\color{gray} 1 1 1 1} 2 2 4 {\color{gray} 1 1 1 1} 2 3 3 4
\\ t = 1: & \cdots {\color{gray} 1 1 1 1 1 1 1 1 1 1 1 1 1 1 1 1 1 1 1 1 1 1 1 1 1 1} 2 4 {\color{gray} 1 1 1} 2 2 4 {\color{gray} 1 1 1} 2 3 3 4 {\color{gray} 1 1 1 1}
\\ t = 2: & \cdots {\color{gray} 1 1 1 1 1 1 1 1 1 1 1 1 1 1 1 1 1 1 1 1 1 1 1 1} 2 4 {\color{gray} 1 1} 2 2 4 {\color{gray} 1 1} 2 3 3 4 {\color{gray} 1 1 1 1 1 1 1 1}
\\ t = 3: & \cdots {\color{gray} 1 1 1 1 1 1 1 1 1 1 1 1 1 1 1 1 1 1 1 1 1} 2 2 4 {\color{gray} 1} 2 2 4 {\color{gray} 1 1} 3 3 4 {\color{gray} 1 1 1 1 1 1 1 1 1 1 1 1}
\\ t = 4: & \cdots {\color{gray} 1 1 1 1 1 1 1 1 1 1 1 1 1 1 1 1 1} 2 2 2 4 {\color{gray} 1 1} 2 4 {\color{gray} 1 1} 3 3 4 {\color{gray} 1 1 1 1 1 1 1 1 1 1 1 1 1 1 1}
\\ t = 5: & \cdots {\color{gray} 1 1 1 1 1 1 1 1 1 1 1 1 1} 2 2 2 4 {\color{gray} 1 1 1 1} 2 4 {\color{gray} 1} 3 3 4 {\color{gray} 1 1 1 1 1 1 1 1 1 1 1 1 1 1 1 1 1 1}
\\ t = 6: & \cdots {\color{gray} 1 1 1 1 1 1 1 1 1} 2 2 2 4 {\color{gray} 1 1 1 1 1} 2 3 4 {\color{gray} 1} 3 4 {\color{gray} 1 1 1 1 1 1 1 1 1 1 1 1 1 1 1 1 1 1 1 1 1}
\\ t = 7: & \cdots {\color{gray} 1 1 1 1 1} 2 2 2 4 {\color{gray} 1 1 1 1 1 1} 2 3 4 {\color{gray} 1 1} 3 4 {\color{gray} 1 1 1 1 1 1 1 1 1 1 1 1 1 1 1 1 1 1 1 1 1 1 1}
\\ t = 8: & \cdots {\color{gray} 1} 2 2 2 4 {\color{gray} 1 1 1 1 1 1 1} 2 3 4 {\color{gray} 1 1 1} 3 4 {\color{gray} 1 1 1 1 1 1 1 1 1 1 1 1 1 1 1 1 1 1 1 1 1 1 1 1 1}
\end{align*}
where a ${\color{gray} 1}$ is considered as a empty bin.
\end{example}

\begin{remark}
Because of our tensor product conventions, our box-ball system propagates in the opposite direction to that of the literature.
\end{remark}

More explicitly, we consider an element $b = \cdots \otimes b_3 \otimes b_2 \otimes b_1$ in a semi-infinite tensor product
\[
B = \cdots \otimes B^{1,s_3} \otimes B^{1,s_2} \otimes B^{1,s_1}
\]
such that $b_j = u_{s_j \varpi_1}$ for all $j \gg 1$. In this paper, we only consider the case when $s_i = 1$ for all $i \gg 1$. Let $\lVert b \rVert$ denote the number entries not equal to $1$ in the (left-to-right) reading word of $b$. The time evolution $T_r^{\infty} \colon B \to B$ of the system is given by the combinatorial $R$-matrix:
\begin{equation}
\label{eq:generic_comb_R}
\begin{split}
B \otimes B^{r,k} & \longrightarrow B^{r,k} \otimes B,
\\ b \otimes u_{k\varpi_r} & \longmapsto \widetilde{u} \otimes \widetilde{b},
\end{split}
\end{equation}
where $k > \lVert b \rVert$. The factor $B^{r,k}$ is called the \defn{carrier}. Hence $T_r^{\infty}(b) = \widetilde{b}$. We note that the action of the carrier in the informal description above is for $r = 1$.

By the Yang--Baxter equation, we can break Equation~\eqref{eq:generic_comb_R} into the following steps:
\begin{align*}
\cdots \otimes b_3 \otimes b_2 \otimes b_1 \otimes c^{(1)}
& \overset{R_1}{\mapstochar\relbar\joinrel\longrightarrow}
\cdots \otimes b_3 \otimes b_2 \otimes c^{(2)} \otimes \widetilde{b}_1
\\ & \overset{R_2}{\mapstochar\relbar\joinrel\longrightarrow}
\cdots \otimes b_3 \otimes c^{(3)} \otimes \widetilde{b}_2 \otimes \widetilde{b}_1
\\ & \overset{R_3}{\mapstochar\relbar\joinrel\longrightarrow}
\cdots \otimes c^{(4)} \otimes \widetilde{b}_3 \otimes \widetilde{b}_2 \otimes \widetilde{b}_1,
\end{align*}
where $c^{(1)} = u_{k\varpi_r}$.
We express this diagrammatically by
\[
\begin{tikzpicture}[xscale=1.1]
\node (b1) at (2, 1) {$b_1$};
\node (b1t) at (2, -1) {$\widetilde{b}_1$};
\node (b2) at (0, 1) {$b_2$};
\node (b2t) at (0, -1) {$\widetilde{b}_2$};
\node (b3) at (-2, 1) {$b_3$};
\node (b3t) at (-2, -1) {$\widetilde{b}_3$};
\node (c1) at (3, 0) {$c^{(1)}$};
\node (c2) at (1, 0) {$c^{(2)}$};
\node (c3) at (-1, 0) {$c^{(3)}$};
\node (c4) at (-3, 0) {$c^{(4)}$};
\draw[->] (b1) -- (b1t);
\draw[->] (b2) -- (b2t);
\draw[->] (b3) -- (b3t);
\draw[->] (c1) -- (c2);
\draw[->] (c2) -- (c3);
\draw[->] (c3) -- (c4);
\draw (-3.6, 0) node {$\cdots$};
\end{tikzpicture}.
\]
We note that there is a well-defined limit since each factor in $b$ is eventually $u_{s_i \varpi_1}$ and the combinatorial $R$-matrix simply acts as $u_{s_i \varpi_1} \otimes u_{k \varpi_r} \mapsto u_{k \varpi_r} \otimes u_{s_i \varpi_1}$ for all $i \gg 1$.

For a state $b = \cdots \otimes b_3 \otimes b_2 \otimes b_1$, a soliton is a pair $\{i < j\}$ such that $\lVert b_{i-1} \rVert = \lVert b_{j+1} \rVert = 0$, $\lVert b_k \lVert > 0$ for all $i \leq k \leq j$, and the reading word $b_i b_{i+1} \dotsm b_j$ is weakly increasing from left-to-right. The \defn{amplitude} or \defn{size} of a soliton is defined as $\lVert b_j \otimes \cdots \otimes b_i \rVert$.

\subsection{Rigged configurations}

Let $B = \bigotimes_{k=1}^m B^{r_k,s_k}$. A \defn{configuration} $\nu = (\nu^{(a)})_{a \in I_0}$ is a sequence of partitions. Define the partition $\mu^{(a)}(B) := \{s_{k_1}, s_{k_2}, \dotsc, s_{k_p}\}$, where $r_k = a$ if and only if $k \in \{k_1, \dotsc, k_p\}$. Define the \defn{vacancy numbers} by
\[
p_{\ell}^{(a)}(\nu, B) := Q_{\ell}\bigl(\mu^{(a)}(B)\bigr) + Q_{\ell}(\nu^{(a-1)}) - 2 Q_{\ell}(\nu^{(a)}) + Q_{\ell}(\nu^{(a+1)}),
\]
where we define the partitions $\nu^{(0)} = \nu^{(n)} = \emptyset$ and
\[
Q_{\ell}(\xi) = \sum_{j=1}^{\infty} \min(\ell, \xi_j)
\]
for the partition $\xi = (\xi_1, \xi_2, \ldots)$. When $\nu$ and $B$ are clear, we simply write $p_{\ell}^{(a)} = p_{\ell}^{(a)}(\nu, B)$.

A \defn{rigged configuration} is a configuration $\nu$ with such that for each part $\nu_i^{(a)}$, we have an integer $J_i^{(a)}$ called a \defn{rigging} such that
\[
L_{\ell}^{(a)} \leq J_i^{(a)} \leq p_{\ell}^{(a)},
\]
where $\ell = \nu_i^{(a)}$ and $L_i^{(a)}$ are the lower bounds given in~\cite[Def.~4.3]{S06}. As we do not use the lower bounds in this paper, we do not recall them here, but instead we refer the reader to~\cite{S06}. We denote a rigged configuration by $(\nu, J)$, where $J = \{ J_i^{(a)} \}_{a \in I_0, i \in \ZZ_{\geq 0}}$ with $J_i^{(a)} = 0$ whenever $\nu_i^{(a)} = 0$. Let $\RC(B)$ denote the set of all rigged configurations.

Next, we recall the bijection $\Phi \colon \RC(B) \to B$. In order to do so, we recall that there are natural crystal morphisms
\begin{align*}
\ls \colon B^{1,s_m} \otimes B^{1,s_{m-1}} \otimes \cdots \otimes B^{1,s_1} & \longrightarrow B^{1,1} \otimes B^{1,s_m-1} \otimes B^{1,s_{m-1}} \otimes \cdots \otimes B^{1,s_1},
\\ b_m \otimes b_{m-1} \otimes \cdots \otimes b_1 & \longmapsto b_{m,1} \otimes b_m' \otimes b_{m-1} \otimes \cdots \otimes b_1,
\end{align*}
where $b_m = b_{m,1} b_{m,2} \dotsm b_{m,s_m}$ and $b' = b_{m,2} \dotsm b_{m,s_m}$, and
\begin{align*}
\operatorname{lb} \colon B^{1,1} \otimes B^{1,s_{m-1}} \otimes \cdots \otimes B^{1,s_1} & \longrightarrow B^{1,s_{m-1}} \otimes \cdots \otimes B^{1,s_1},
\\ b_m \otimes b_{m-1} \otimes \cdots \otimes b_1 & \longmapsto b_{m-1} \otimes \cdots \otimes b_1.
\end{align*}
The bijection $\Phi \colon \RC(B) \to B$ is defined recursively using the following commuting diagrams:
\begin{align*}
\xymatrixrowsep{3pc}
\xymatrixrowsep{3.5pc}
\xymatrix{\RC(B^{1,s_m} \otimes \cdots \otimes B^{1,s_1}) \ar[r]^-{\Phi} \ar[d]_{\gamma}
& B^{1,s_m} \otimes \cdots \otimes B^{1,s_1} \ar[d]^{\ls}
\\ \RC(B^{1,1} \otimes B^{1,s_m-1} \otimes \cdots \otimes B^{1,s_1}) \ar[r]_-{\Phi}
& B^{1,1} \otimes B^{1,s_m-1} \otimes \cdots \otimes B^{1,s_1}}
&\raisebox{-30pt}{($s_m > 1$)}
\\ \\
\xymatrixrowsep{3pc}
\xymatrixrowsep{3.5pc}
\xymatrix{\RC(B^{1,1} \otimes B^{1,s_{m-1}} \otimes \cdots \otimes B^{1,s_1}) \ar[r]^-{\Phi} \ar[d]_{\delta}
& B^{1,1} \otimes B^{1,s_{m-1}} \otimes \cdots \otimes B^{1,s_1} \ar[d]^{\operatorname{lb}}
\\ \RC(B^{1,s_{m-1}} \otimes \cdots \otimes B^{1,s_1}) \ar[r]_-{\Phi}
& B^{1,s_{m-1}} \otimes \cdots \otimes B^{1,s_1}}
&
\end{align*}
with $\gamma$ be the identity map and $\delta$ being defined as follows.
Determine the minimal $1 \leq \ell^{(1)} \leq \ell^{(2)} \leq \cdots \leq \ell^{(n-1)}$ such that for any $\ell^{(a)} < \infty$, there exists an $i^{(a)}$ such that $\ell^{(a)} = \nu_{i^{(a)}}^{(a)}$ and $J_{i^{(a)}}^{(a)} = p_{\ell^{(a)}}^{(a)}$ and $\ell^{(a)} = \infty$ if no such $i^{(a)}$ exists. Define $\delta(\nu, J) = (\widetilde{\nu}, \widetilde{J})$ as the configuration $\widetilde{\nu}$ obtained by removing a box from the $i^{(a)}$-th row if $\ell^{(a)} < \infty$ and riggings
\[
\widetilde{J}_i^{(a)} = \begin{cases}
J_i^{(a)} & \text{if } i \neq i^{(a)},\\
p_{\ell^{(a)}}^{(a)}(\widetilde{\nu}_{i^{(a)}}^{(a)}, B^{1,s_{m-1}} \otimes \cdots \otimes B^{1,s_1}) & \text{if } i = i^{(a)},
\end{cases}
\]
where we consider $i^{(a)} = \infty$ if $\ell^{(a)} = \infty$.

We say a row $\nu_i^{(a)} = \ell$ is \defn{singular} if $J_i^{(a)} = p_{\ell}^{(a)}$. The map $\delta$ can roughly be described as removing a box from a singular row in $\nu^{(a)}$ of weakly increasing length as $a$ increases and change the corresponding riggings such that they remain singular in the resulting rigged configuration. Define the \defn{return value} of $\delta$ as being the smallest $a$ such that $\ell^{(a-1)} < \infty$. The bijection can then be described as the left-to-right reading word of $B^{1,s_m} \otimes \cdots \otimes B^{1,s_1}$ given by the return values of each application of $\delta$.

For simplicity of (hand) computations, we can perform $\gamma$ and $\delta$ as one operation where we require $s_m \leq \ell^{(1)} \leq \cdots \leq \ell^{(n-1)}$.

\begin{example}
Let $B = B^{1,3} \otimes B^{1,1} \otimes B^{1,2}$ of type $\asl_4$. Consider the rigged configuration
\[
(\nu, J) =
\raisebox{15pt}{$
\begin{array}[t]{r|c|c|c|l}
 \cline{2-4} -1 &\phantom{|}&\phantom{|}&\phantom{|}& -1 \\
 \cline{2-4} -1 &\phantom{|}&\phantom{|}& \multicolumn{2 }{l}{ -1 } \\
 \cline{2-3}
\end{array}
\quad
\begin{array}[t]{r|c|c|c|c|l}
 \cline{2-5} -1 &\phantom{|}&\phantom{|}&\phantom{|}&\phantom{|}& -1 \\
 \cline{2-5}
\end{array}
\quad
\begin{array}[t]{r|c|c|l}
 \cline{2-3} -2 &\phantom{|}&\phantom{|}& -2 \\
 \cline{2-3}
\end{array}
$}.
\]
We perform the bijection $\Phi$ where we mark with a $\times$ the boxes removed under $\delta \circ \gamma$:
\begin{gather*}
\begin{array}[t]{r|c|c|c|l}
 \cline{2-4} -1 &\phantom{|}&\phantom{|}&\markbox& -1 \\
 \cline{2-4} -1 &\phantom{|}&\phantom{|}& \multicolumn{2 }{l}{ -1 } \\
 \cline{2-3}
\end{array}
\quad
\begin{array}[t]{r|c|c|c|c|l}
 \cline{2-5} -1 &\phantom{|}&\phantom{|}&\phantom{|}&\markbox& -1 \\
 \cline{2-5}
\end{array}
\quad
\begin{array}[t]{r|c|c|l}
 \cline{2-3} -2 &\phantom{|}&\phantom{|}& -2 \\
 \cline{2-3}
\end{array}
\\[1em]
\begin{array}[t]{r|c|c|l}
 \cline{2-3} -1 &\phantom{|}&\phantom{|}& -1 \\
 \cline{2-3} -1 &\phantom{|}&\markbox& -1 \\
 \cline{2-3}
\end{array}
\quad
\begin{array}[t]{r|c|c|c|l}
 \cline{2-4} 0 &\phantom{|}&\phantom{|}&\markbox& 0 \\
 \cline{2-4}
\end{array}
\quad
\begin{array}[t]{r|c|c|l}
 \cline{2-3} -2 &\phantom{|}&\phantom{|}& -2 \\
 \cline{2-3}
\end{array}
\\[1em]
\begin{array}[t]{r|c|c|l}
 \cline{2-3} 0 &\phantom{|}&\phantom{|}& -1 \\
 \cline{2-3} 0 &\markbox& \multicolumn{1}{l}{ 0 } \\
 \cline{2-2}
\end{array}
\quad
\begin{array}[t]{r|c|c|l}
 \cline{2-3} 1 &\phantom{|}&\markbox& 1 \\
 \cline{2-3}
\end{array}
\quad
\begin{array}[t]{r|c|c|l}
 \cline{2-3} -2 &\phantom{|}&\markbox& -2 \\
 \cline{2-3}
\end{array}
\\[1em]
\begin{array}[t]{r|c|c|l}
 \cline{2-3} 0 &\phantom{|}&\phantom{|}& -1 \\
 \cline{2-3}
\end{array}
\quad
\begin{array}[t]{r|c|l}
 \cline{2-2} 0 &\phantom{|}& 0 \\
 \cline{2-2}
\end{array}
\quad
\begin{array}[t]{r|c|l}
 \cline{2-2} -1 &\phantom{|}& -1 \\
 \cline{2-2}
\end{array}
\\[1em]
\begin{array}[t]{r|c|c|l}
 \cline{2-3} -1 &\phantom{|}&\markbox& -1 \\
 \cline{2-3}
\end{array}
\quad
\begin{array}[t]{r|c|l}
 \cline{2-2} 0 &\phantom{|}& 0 \\
 \cline{2-2}
\end{array}
\quad
\begin{array}[t]{r|c|l}
 \cline{2-2} -1 &\phantom{|}& -1 \\
 \cline{2-2}
\end{array}
\\[1em]
\begin{array}[t]{r|c|l}
 \cline{2-2} 0 &\markbox& 0 \\
 \cline{2-2}
\end{array}
\quad
\begin{array}[t]{r|c|l}
 \cline{2-2} 0 &\markbox& 0 \\
 \cline{2-2}
\end{array}
\quad
\begin{array}[t]{r|c|l}
 \cline{2-2} -1 &\markbox& -1 \\
 \cline{2-2}
\end{array}
\end{gather*}
Therefore, the result is
\[
\Phi(\nu, J) = \young(334) \otimes \young(1) \otimes \young(24).
\]
\end{example}

The inverse map $\delta^{-1}$ can roughly be described as adding a box to singular row of $\nu^{(a)}$ such that the length is weakly decreasing as $a$ decreases from $b$ to $1$, where $b$ is the value we added to the reading word, and change the corresponding riggings such that they remain singular in the resulting rigged configuration. For $\Phi^{-1}$, we add letters following $\delta^{-1}$ of the reading word going from right-to-left.

It is known that there exists a $U_q'(\asl_n)$-crystal structure on $\RC(B)$~\cite{S06, SW10}. Moreover, by combining the results of~\cite{DS06, KSS02, SW10}, we have the following.

\begin{thm}
Let $B = \bigotimes_{i=1}^m B^{1,s_i}$ be a tensor product of KR crystals of type $\asl_n$. Then
\[
\Phi \colon \RC(B) \to B
\]
is a $U_q'(\asl_n)$-crystal isomorphism. Moreover, rigged configurations are invariant under the combinatorial $R$-matrix: $R(\nu, J) = (\nu, J)$.
\end{thm}

We note some additional key properties of the bijection $\Phi$.
Consider some $b \in B = \bigotimes_{i=1}^m B^{1,s_i}$, and let $(\nu, J) = \Phi^{-1}(b)$.
Then we have $\Phi^{-1}(u_{s \varpi_r} \otimes b) = (\nu, J)$ for any $r \in I_0$ and $s \in \ZZ_{>0}$ from the definition of $\delta^{-1}$.
Note that $\Phi^{-1}(u_{s \varpi_r})$ is the empty rigged configuration.
Therefore, we also have $\Phi^{-1}(b \otimes u_{s \varpi_r}) = (\nu, \widetilde{J})$, where
\[
\widetilde{J}_i^{(a)} = \begin{cases} J_i^{(a)} + \min(s, \nu_i^{(a)}) & \text{if } a = r, \\ J_i^{(a)} & \text{otherwise,} \end{cases}
\]
from the definition of $\delta^{-1}$ and that adding the factor $B^{r,s}$ increases the vacancy numbers $p_{\ell}^{(r)}$ by $\min(s, \ell)$.

Now consider a state $b$ of a box-ball system $B$.
From the above properties, it is clear that $\Phi^{-1}(b)$ is well-defined.
Moreover, from the description of $\Phi$, we have $\lvert \nu^{(1)} \rvert = \lVert b \rVert$.
In fact, it was shown in~\cite{KOSTY06} that lengths of each row in $\nu^{(1)}$ correspond to the sizes of the solitions (except possibly when they are interacting).
Furthermore, to perform the time evolution $T_r^{\infty}$, we add a carrier to the right, which increases the riggings $J_i^{(r)}$ by $\nu_i^{(r)}$ for all $i$, and we denote this new rigged configuration by $(\nu, \widetilde{J})$.
Since rigged configurations are invariant under the combinatorial $R$-matrix, we have
\[
\Phi^{-1}\bigl(T^{\infty}_r(b)\bigr) = (\nu, \widetilde{J})
\]
for any $r \in I_0$ (we can remove the carrier from the left because it is $u_{k \varpi_r}$).
Thus rigged configurations completely describe the action-angle variables of the associated soliton cellular automaton.

\begin{remark}
In~\cite{KOSTY06}, the theory of vertex operators was used to show that rigged configurations give the action-angle variables and describe how the rigged configuration behave under time evolution.
\end{remark}

\begin{example}
Define $b^{(t)} = T_1^{\infty}(b^{(t-1)})$ with $b^{(0)}$ being the initial state from Example~\ref{ex:box_ball_dynamics}. Then in type $\asl_4$, we have
\[
\Phi^{-1}(b^{(t)}) =
\raisebox{20pt}
{$
\begin{array}[t]{r|c|c|c|c|l}
 \cline{2-5} &\phantom{|}&\phantom{|}&\phantom{|}&\phantom{|}& -1 + 4t \\
 \cline{2-5} &\phantom{|}&\phantom{|}&\phantom{|}& \multicolumn{2 }{l}{ 3 + 3t } \\
 \cline{2-4} &\phantom{|}&\phantom{|}& \multicolumn{3 }{l}{ 9 + 2t } \\
 \cline{2-3}
\end{array}
\quad
\begin{array}[t]{r|c|c|c|l}
 \cline{2-4} &\phantom{|}&\phantom{|}&\phantom{|}& -2 \\
 \cline{2-4} &\phantom{|}&\phantom{|}& \multicolumn{2 }{l}{ -1 } \\
 \cline{2-3}
\end{array}
\quad
\begin{array}[t]{r|c|c|c|l}
 \cline{2-4} &\phantom{|}&\phantom{|}&\phantom{|}& -1 \\
 \cline{2-4}
\end{array}
$},
\]
where we omit the vacancy numbers since $p_{\ell}^{(1)} = \infty$, for all $\ell \in \ZZ_{> 0}$, in the limit as the number of factors goes to infinity ({\it i.e.}, as we take $m \to \infty$).
\end{example}

\subsection{Formulas for shapes}

We recall the results of~\cite{LPS15}.
For the remainder of this section, we fix $B = \bigotimes_{i=1}^m B^{1,s_i}$ in type $\asl_n$.

We first define a set of commuting variables $\bfX = \{x_i^{(a)} \mid 1 \leq a \leq n, 1 \leq i \leq m\}$, and we consider the upper indices modulo $n$.
An \defn{$a$-cylindric semistandard skew tableau} of skew shape $\lambda / \mu$, where $\mu \subseteq \lambda$ are partitions, is a filling $\mcT \colon \lambda / \mu \to \{1, 2, \dotsc, m\}$, with $\mcT(i,j)$ being the $i$-th row from the top and $j$-th column from the left, such that $\mcT(i, j) < \mcT(i+1, j)$ and $\mcT(i,j) \leq \mcT(i, j+1)$, where we consider the cells of $\lambda / \mu$ as lying in the cylinder $\ZZ^2 / (n-a, a)$.

That is to say, an $a$-cylindric semistandard skew tableau is a usual semistandard skew tableaux except we have some extra conditions along the boundary. These can be realized by taking a tiling of the tableaux $\mcT$ with shifts by $(n-a, a)$ such that the result is semistandard. We call this infinite shape the \defn{universal cover} of $\mcT$.

\begin{example}
Let $n = 3$. The (skew) tableaux
\[
\mcT = \young(13,2)
\]
is semistandard in the usual sense, where we take $\lambda = 21$ and $\mu = \emptyset$.
However, it is not a $1$-cylindric semistandard (skew) tableaux as the universal cover of $\mcT$
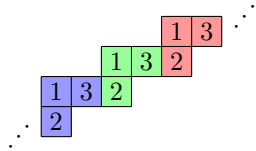
\[
\begin{tikzpicture}[scale=0.4]
\fill[color=red, opacity=0.4] (2,1) rectangle (4,2);
\fill[color=red, opacity=0.4] (2,0) rectangle (3,1);
\fill[color=green, opacity=0.4] (0,0) rectangle (2,1);
\fill[color=green, opacity=0.4] (0,-1) rectangle (1,0);
\fill[color=blue, opacity=0.4] (-2,-1) rectangle (0,0);
\fill[color=blue, opacity=0.4] (-2,-2) rectangle (-1,-1);
\draw (2,2) -- (4,2);
\draw (0,1) -- (4,1);
\draw (-2,0) -- (3,0);
\draw (-2,-1) -- (1,-1);
\draw (-2,-2) -- (-1,-2);
\draw (4,2) -- (4,1);
\draw (3,2) -- (3,0);
\draw (2,2) -- (2,0);
\draw (1,1) -- (1,-1);
\draw (0,1) -- (0,-1);
\draw (-1,0) -- (-1,-2);
\draw (-2,0) -- (-2,-2);
\draw (2.5, 1.5) node {$1$};
\draw (3.5, 1.5) node {$3$};
\draw (0.5, 0.5) node {$1$};
\draw (1.5, 0.5) node {$3$};
\draw (2.5, 0.5) node {$2$};
\draw (-1.5, -0.5) node {$1$};
\draw (-0.5,- 0.5) node {$3$};
\draw (0.5, -0.5) node {$2$};
\draw (-1.5, -1.5) node {$2$};
\draw (4.5, 2) node[rotate=90] {$\ddots$};
\draw (-3, -2) node[rotate=90] {$\ddots$};
\end{tikzpicture}
\]
is not semistandard (in the usual sense).
All other semistandard tableaux of shape $21$ are also $1$-cylindric semistandard.
\end{example}

The \defn{loop Schur function} of skew shape $\lambda / \mu$ is
\begin{equation}
\label{eq:defn_loop_schur}
s_{\lambda / \mu}^{(k)}(\bfX) = \sum_{\mcT} \prod_{\mathfrak{s} \in \lambda / \mu} x_{\mcT(\mathfrak{s})}^{(c(\mathfrak{s}) + k)},
\end{equation}
where we sum over all semistandard tableau $\mcT$ of shape $\lambda / \mu$ with max entry $m$ and $c(\mathfrak{s}) = i - j$ is the content of the cell $\mathfrak{s} = (i, j)$. Note that our notion of content is the negative of the usual one.
The \defn{cylindric loop Schur function} $\overline{s}_{\lambda / \mu,a}^{(k)}$ is given by Equation~\eqref{eq:defn_loop_schur} except the sum is over all \emph{$a$-cylindric} semistandard tableau $\mcT$.

Let $\mathbf{x}_i = \{x_i^{(a)} \mid 1 \leq a \leq n\}$.
Define the \defn{birational $R$-matrix}, or geometric $R$-matrix, as the rational map
\[
R \colon \QQ(x_1^{(1)}, \dotsc, x_1^{(n)}, x_2^{(1)}, \dotsc, x_2^{(n)}) \to \QQ(x_2^{(1)}, \dotsc, x_2^{(n)}, x_1^{(1)}, \dotsc, x_1^{(n)})
\]
given on generators by
\[
R(x_1^{(a)}) = x_2^{(a-1)} \frac{\kappa_{a-1}(\bfx_1, \bfx_2)}{\kappa_a(\bfx_1, \bfx_2)},
\qquad\qquad
R(x_2^{(a)}) = x_1^{(a+1)} \frac{\kappa_{a+1}(\bfx_1, \bfx_2)}{\kappa_a(\bfx_1, \bfx_2)},
\]
where
\[
\kappa_a(\bfx_1, \bfx_2) = \sum_{j=a}^{a+n-1} \prod_{k=a+1}^j x_1^{(k)} \prod_{k=j+1}^{a+n-1} x_2^{(k)}.
\]
We then extend this to $R_k \colon \QQ(\bfX) \to \QQ(\bfX)$ by acting as $R$ on $\bfx_k$ and $\bfx_{k+1}$ and fixing $\bfx_i$ for $i \neq k, k+1$.

\begin{thm}[{\cite{LPS15}}]
A cylindric loop Schur function is in the ring of loop symmetric polynomials $\ZZ[\mathbf{E}]$, where $\mathbf{E} = \{e_i^{(a)} \mid 1 \leq a \leq n, 1 \leq i \leq m\}$ with
\[
e_k^{(a)} := \sum_{1 \leq i_1 < i_2 < \cdots < i_k \leq m} x_{i_1}^{(a)} x_{i_2}^{(a+1)} \cdots x_{i_k}^{(a+k-1)}.
\]
Moreover, the cylindric loop Schur functions are invariant under the birational $R$-matrix.
\end{thm}

Next, define a partition $\lambda(a, i)$ recursively by $\lambda(a, 0) := (n - a)^m$, and then $\lambda(a, i+1)$ is given by removing a ribbon strip of size at most $n$ starting from the bottom-left corner from $\lambda(a, i)$. By ribbon strip, we mean that $\lambda(a, i) / \lambda(a, i+1)$ is connected and contains no $2 \times 2$ squares.

Let $\trop$ denote the tropicalization functor, where we essentially replace $+$ with $\min$ and $\cdot$ with $+$. More explicitly, for a polynomial $p(z) = \sum_j c_j z^j$ with $c_j \in \RR_{\geq 0}$, we have
\[
(\trop p)(z) = \min_j \{0 + j \cdot z\}.
\]
For a precise definition, we refer the reader to~\cite{BK00, BK07}.

The following conjecture was given in~\cite{LPS15}.

\begin{conj}
\label{conj:formula_shapes}
Let $p = b_m \otimes \dotsm \otimes b_1 \in B^{1,s_m} \otimes \dotsm \otimes B^{1,s_1}$ of type $\asl_n$. Define $x_j^{(a+j-1)} = \varphi_a(b_{m+1-j})$, {\it i.e.}, the number of occurrences of $a$ in $b_{m+1-j}$. Then for $1 \leq a \leq n - 1$ and $i \in \ZZ_{> 0}$, we have
\[
\trop\left(\frac{\overline{s}^{(0)}_{\lambda(a, i-1),a}}{\overline{s}^{(0)}_{\lambda(a, i),a}}\right)(\bfX) = \nu_i^{(a)},
\]
where $(\nu, J) = \Phi^{-1}(p)$.
\end{conj}

The following special case is known.

\begin{thm}[{\cite{LPS15}}]
Conjecture~\ref{conj:formula_shapes} holds for $a = 1$ and all $i \in \ZZ_{> 0}$.
\end{thm}


\section{Conjectural formula}
\label{sec:conjectures}

In this section, we give our conjectural formula for the riggings under the KSS bijection.

Fix some $a \in I_0$. Consider an initial state $b \in B = \bigotimes_{i=1}^m B^{1,s_i}$ of type $\asl_n$. Let $b_t = (T_a^{\infty})^t(b)$ be the current state after $t$ steps.
We construct a skew partition $\eta^{[i]}_{a,t} = \lambda^{[i]}_{a,t} / \mu^{[i]}_{a,t}$ as follows.
Start with the infinite shape formed from $\lambda(a, i-1)$, considered as contained in an $m \times (n - a)$ rectangle, shifted by $a$ at each step.
Then we make a cut at $i$ rows up from the bottom of some rectangle $r$ and then $i$ columns to the left of the rectangle at $t+1$ copies from $r$.
Lastly, superimpose a copy of $\lambda(a, i)$ to the upper right and bottom left rectangles aligned in the corners.

Next, we define $\widetilde{\eta}^{[i]}_{a,t}$ as the shape given by $\widetilde{t}$ copies of $\lambda(a, i-1)$, where $\widetilde{t}$ equals the number of copies of $\lambda(a, i-1)$ in $\eta^{[i]}_{a,t}$ (alternatively, the length of the first row of $\lambda^{[i]}_{a,t}$ divided by $n-a$).

\begin{example}
Consider $B = \bigotimes_{i=1}^6 B^{1,s_i}$ of type $\asl_7$. We have $\eta^{[1]}_{3,2}$ given by the shape filled in gray in Figure~\ref{fig:example_shape}. We also have $\widetilde{\eta}^{[1]}_{3,2} = 16^2 14^3 11^3 7^3 3^3 / 12^3 8^3 4^3$. Additionally, we have $\eta^{[2]}_{3,2}$ as the shape in Figure~\ref{fig:example_shape_strip}
and $\widetilde{\eta}^{[2]}_{3,2} = 16^1 14^3 10^3 6^3 2^3 / 12^3 8^3 4^3$.
\end{example}

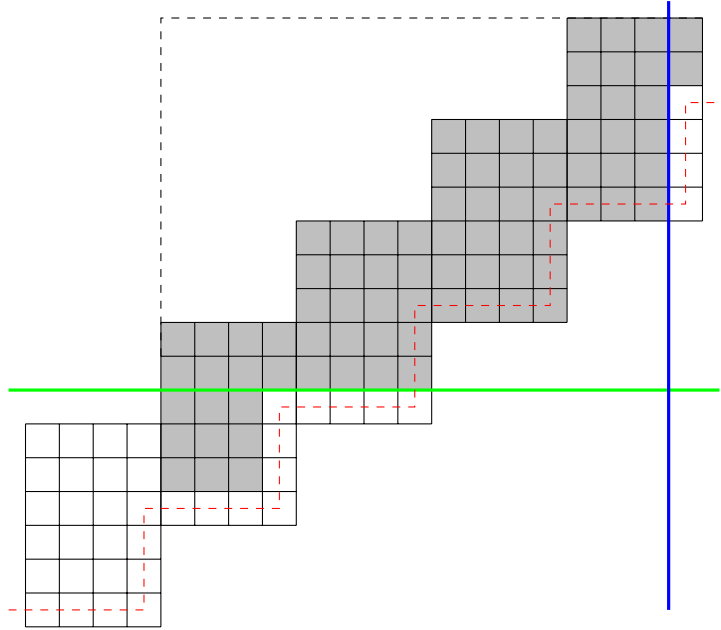
\begin{figure}
\[
\begin{tikzpicture}[scale=.45]
\fill[color=lightgray] (4,6) rectangle (8,8);
\fill[color=lightgray] (4,3) rectangle (7,6);
\fill[color=lightgray] (8,6) rectangle (12,11);
\fill[color=lightgray] (12,8) rectangle (16,14);
\fill[color=lightgray] (16,11) rectangle (19,17);
\fill[color=lightgray] (19,15) rectangle (20,17);
\draw[step=1] (0,-1) grid (4,5);
\draw[step=1] (4,2) grid (8,8);
\draw[step=1] (8,5) grid (12,11);
\draw[step=1] (12,8) grid (16,14);
\draw[step=1] (16,11) grid (20,17);
\draw[dashed] (4,7) -- (4,17) -- (20,17);
\draw[green, very thick] (-.5, 6) -- (20.5, 6);
\draw[blue, very thick] (19, -.5) -- (19, 17.5);
\draw[red, dashed] (-0.5, -0.5) -- (3.5, -0.5) -- (3.5, 2.5) -- (7.5, 2.5) -- (7.5, 5.5) -- (11.5, 5.5) -- (11.5, 8.5) -- (15.5, 8.5) -- (15.5, 11.5) -- (19.5, 11.5) -- (19.5, 14.5) -- (20.5, 14.5);
\end{tikzpicture}
\]
\caption{The shape $\lambda^{[1]}_{3,2}$ given by the dashed lines and the gray boxes. The shape $\mu^{[1]}_{3,2}$ is given by the region between the dashed lines and the gray boxes.}
\label{fig:example_shape}
\end{figure}

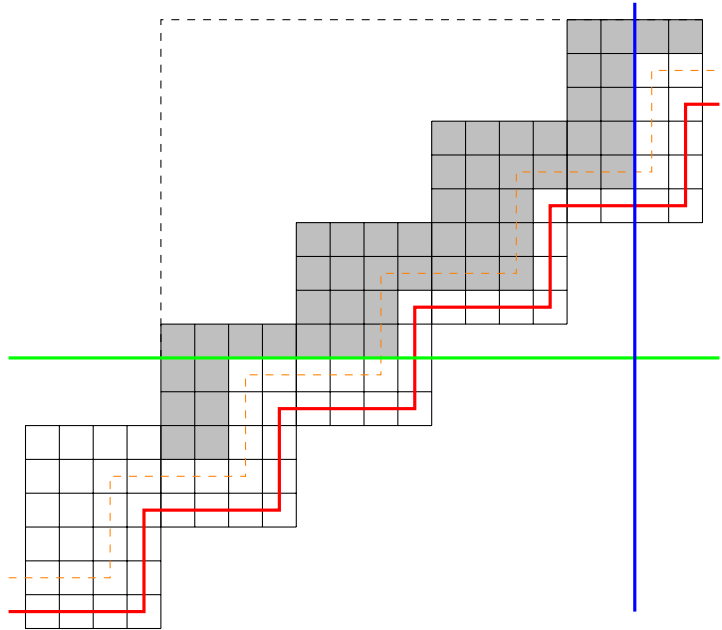
\begin{figure}
\[
\begin{tikzpicture}[scale=.45]
\fill[color=lightgray] (4,4) rectangle (6,7);
\fill[color=lightgray] (4,7) rectangle (8,8);
\fill[color=lightgray] (8,7) rectangle (11,11);
\fill[color=lightgray] (11,9) rectangle (12,11);
\fill[color=lightgray] (12,9) rectangle (15,14);
\fill[color=lightgray] (15,12) rectangle (16,14);
\fill[color=lightgray] (16,12) rectangle (18,17);
\fill[color=lightgray] (18,16) rectangle (20,17);
\draw[step=1] (0,-1) grid (4,5);
\draw[step=1] (4,2) grid (8,8);
\draw[step=1] (8,5) grid (12,11);
\draw[step=1] (12,8) grid (16,14);
\draw[step=1] (16,11) grid (20,17);
\draw[dashed] (4,7) -- (4,17) -- (20,17);
\draw[red, very thick] (-0.5, -0.5) -- (3.5, -0.5) -- (3.5, 2.5) -- (7.5, 2.5) -- (7.5, 5.5) -- (11.5, 5.5) -- (11.5, 8.5) -- (15.5, 8.5) -- (15.5, 11.5) -- (19.5, 11.5) -- (19.5, 14.5) -- (20.5, 14.5);
\draw[orange, dashed] (-0.5, 0.5) -- (2.5, 0.5) -- (2.5, 3.5) -- (6.5, 3.5) -- (6.5, 6.5) -- (10.5, 6.5) -- (10.5, 9.5) -- (14.5, 9.5) -- (14.5, 12.5) -- (18.5, 12.5) -- (18.5, 15.5) -- (20.5, 15.5);
\draw[green, very thick] (-.5, 7) -- (20.5, 7);
\draw[blue, very thick] (18, -.5) -- (18, 17.5);
\end{tikzpicture}
\]
\caption{The skew shape $\eta^{[2]}_{3,2}$ shaded in gray.}
\label{fig:example_shape_strip}
\end{figure}

Next, in order to have the upper-left corner of each box to have color $n$, we need to perform a shift of the content by $-\ell(\mu_{a,t}^{[i]}) \pmod{n}$, where $\ell(\xi)$ is the length of $\xi$.

\begin{conj}
\label{conj:formula_riggings}
Fix some $a \in I_0$. Let $b$ be an initial state of type $\asl_n$ such that all values not equal to $1$ occur in the last $m$ factors.
Let $b_t = (T_a^{\infty})^t(b)$, and let $(\nu, J) = \Phi^{-1}(b_t)$.
There exists some $T \geq 0$ such that
\[
J_i^{(a)} = \trop\left(\frac{s^{(\ell)}_{\eta_{a,t}^{[i]}}}{s^{(\ell)}_{\widetilde{\eta}_{a,t}^{[i]}}}\right)(\bfX),
\]
where $\ell = -\ell(\mu_{a,t}^{[i]}) \pmod{n}$,
for all $t \geq T$.
\end{conj}

The following lemma was proven by Pavlo Pylyavskyy.

\begin{lemma}
\label{lemma:single_period}
Fix a cylindric shape $\theta$, and let $\xi$ denote the shape given by $t$ repetitions of $\theta$.
Let $\mcT'$ denote cylindric semistandard tableau of shape $\xi$. Then we have
\[
\wt(\mcT') \geq t \cdot \min_{\mcT} \wt(T),
\]
where the minimum is taken over all cylindric semistandard tableaux $\mcT$ of shape $\theta$.
\end{lemma}

\begin{proof}
Consider the tiling $\tau$ of the universal cover of $\xi$ with $\mcT'$.
Let $\tau_i$ be the shift of $\tau$ by $i$ periods of $\theta$, so that $\tau_0 = \tau$.
Note the shape of $\tau_i$ is the same as the shape of $\tau_j$ for all $i,j$, just potentially has a different filling, and that $\tau_{i+t} = \tau_i$ because there are exactly $t$ copies of $\theta$ in $\xi$.

Now, superimpose on top of each other $\{\tau_1, \tau_2, \dotsc, \tau_t\}$.
We can think that each cell of the universal cover of $\xi$, which is the same as universal cover of $\theta$, contains exactly $t$ numbers now: the fillings of this cell's $t$ shifts by the period.

Here is a key observation. Consider cells $\mathfrak{s}$ and $\mathfrak{s}'$ such that $\mathfrak{s}'$ is directly to the right of $\mathfrak{s}$.
Let $a_j = \tau_j(\mathfrak{s})$ and $a'_j = \tau_j(\mathfrak{s}')$.
Thus by the semistandard condition, we have $a_j \leq a'_j$ for $1 \leq j \leq t$.
Let
\begin{align*}
\{b_1 \leq b_2 \leq \cdots \leq b_t\} & = \{a_1, a_2, \dotsc, a_t\},
\\ \{b'_1 \leq b'_2 \leq \cdots \leq b'_t\} & = \{a'_1, a'_2, \dotsc, a'_t\}.
\end{align*}
It is easy to see that those new numbers also satisfy $b_j \leq b'_j$ for all $1 \leq j \leq t$.

The analogous statement is true if $\mathfrak{s}'$ is directly below $\mathfrak{s}$ and we replace weak inequalities by strict.
Let $\mcT_i$ be the filling of shape $\theta$ with numbers $b_i$, {\it i.e.} for each cell $\mathfrak{s}$ we fill it with the $i$-th smallest number among $\{a_1, \dotsc, a_t\}$.
This is well defined since all translations of $\mathfrak{s}$ have the same collection of numbers in them, one from each $\tau_i$.

The previous discussion implies that each $\mcT_i$ is a cylindric semistandard tableau of shape $\theta$.
On the other hand, it is clear that $\wt(\mcT') = \sum_{i=1}^t \wt(\mcT_i)$.
Since $\wt(\mcT_i) \geq \wt(\mcT)$ by definition, the statement follows.
\end{proof}

\begin{thm}
\label{thm:eventually_stabilizes}
Fix some $a \in I_0$. Let $b$ be an initial state of type $\asl_n$ such that all values not equal to $1$ occur in the last $m$ factors.
There exists some $T \geq 0$ such that
\[
\trop\left(\frac{s^{(\ell)}_{\eta_{a,t}^{[i]}} s^{(\ell)}_{\widetilde{\eta}_{a,t+1}^{[i]}}}{s^{(\ell)}_{\widetilde{\eta}_{a,t}^{[i]}} s^{(\ell)}_{\eta_{a,t+1}^{[i]}}}\right)(\bfX)
= \trop\left(\frac{\overline{s}^{(0)}_{\lambda(a, i-1),a}}{\overline{s}^{(0)}_{\lambda(a, i),a}}\right)(\bfX),
\]
where $\ell = -\ell(\mu_{a,t}^{[i]}) \pmod{n}$,
for all $t \geq T$.
\end{thm}

\begin{proof}
Let $\mcT$ be the minimal weight tableau from Lemma~\ref{lemma:single_period}. Let $\mathcal{P}$ be the periodic result.
Let $k = \wt(\mathcal{P}) - \wt(\mcT)$. If $k > 0$, there exists some $t \geq 0$ such that $B_{\mathcal{P}} + t \wt(\mathcal{P}) \geq B_{\mcT} + t \wt(\mcT)$, where $B_{\mathcal{S}}$ is the base cost of the periodic part $\mathcal{S}$.

If $k = 0$, then by the proof of Lemma~\ref{lemma:single_period}, we have that all minimal elements have the same weight.
In particular, we can take the lex largest tableaux $\mcT$ since they all have the same weight.
Thus we can replace $\mathcal{P}$ by copies of $\mcT$ with the same base and obtain a smaller weight.
\end{proof}

We note that Theorem~\ref{thm:eventually_stabilizes} states that the degree eventually stabilizes. Moreover, it shows that Conjecture~\ref{conj:formula_riggings} holds up to a constant that does not depend on $t \gg 1$.

\begin{example}
\label{ex:sl2}
Consider type $\asl_2$ and $m = 3$. We have
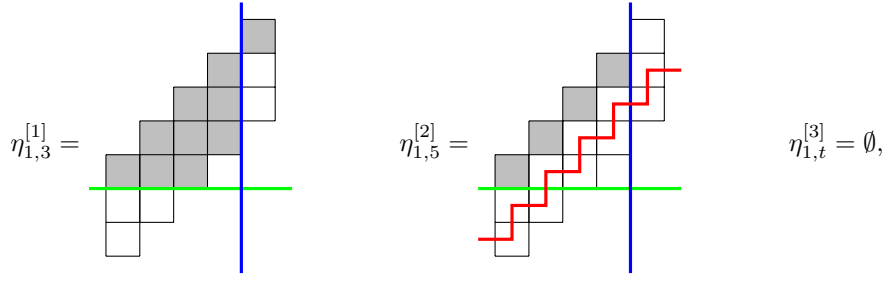
\[
\eta_{1,3}^{[1]} =
\begin{tikzpicture}[scale=.45, baseline=40]
\fill[color=lightgray] (0,2) rectangle (1,3);
\fill[color=lightgray] (1,2) rectangle (2,4);
\fill[color=lightgray] (2,2) rectangle (3,5);
\fill[color=lightgray] (3,3) rectangle (4,6);
\fill[color=lightgray] (4,6) rectangle (5,7);
\draw[step=1] (0,0) grid (1,3);
\draw[step=1] (1,1) grid (2,4);
\draw[step=1] (2,2) grid (3,5);
\draw[step=1] (3,3) grid (4,6);
\draw[step=1] (4,4) grid (5,7);
\draw[green, very thick] (-.5, 2) -- (5.5, 2);
\draw[blue, very thick] (4, -.5) -- (4, 7.5);
\end{tikzpicture}
\qquad\qquad
\eta_{1,5}^{[2]} =
\begin{tikzpicture}[scale=.45, baseline=40]
\fill[color=lightgray] (0,2) rectangle (1,3);
\fill[color=lightgray] (1,3) rectangle (2,4);
\fill[color=lightgray] (2,4) rectangle (3,5);
\fill[color=lightgray] (3,5) rectangle (4,6);
\draw[step=1] (0,0) grid (1,3);
\draw[step=1] (1,1) grid (2,4);
\draw[step=1] (2,2) grid (3,5);
\draw[step=1] (3,3) grid (4,6);
\draw[step=1] (4,4) grid (5,7);
\draw[green, very thick] (-.5, 2) -- (5.5, 2);
\draw[blue, very thick] (4, -.5) -- (4, 7.5);
\draw[red, very thick] (-0.5, 0.5) -- (0.5, 0.5) -- (0.5, 1.5) -- (1.5, 1.5) -- (1.5, 2.5) -- (2.5, 2.5) -- (2.5, 3.5) -- (3.5, 3.5) -- (3.5, 4.5) -- (4.5, 4.5) -- (4.5, 5.5) -- (5.5, 5.5);
\end{tikzpicture}
\qquad\qquad
\eta_{1,t}^{[3]} = \emptyset,
\]
for all $t > 0$.
We consider the (generic) element $b = 1^{\alpha} 2^{\beta} \otimes 1^{\gamma} 2^{\delta} \otimes 1^{\zeta} 2^{\theta}$, so we have
\[
\begin{bmatrix}
x_1^{(1)} & x_2^{(2)} & x_3^{(1)} \\ x_1^{(2)} & x_2^{(1)} & x_3^{(2)}
\end{bmatrix}
=
\begin{bmatrix}
\alpha & \gamma & \epsilon \\ \beta & \delta & \zeta
\end{bmatrix}
\]
Thus, we have
\begin{align*}
\frac{s_{\eta_{1,3}^{[1]}}^{(3)}}{s_{\widetilde{\eta}_{1,3}^{[1]}}^{(3)}} & = \frac{\beta^4 \delta^3 \zeta^2 + \beta^4 \delta^2 \epsilon \zeta^2 + \beta^3 \delta^3 \gamma \zeta^2 + \beta^3 \delta^2 \gamma \epsilon \zeta^2 + \beta^3 \delta^2 \epsilon \zeta^3 + \beta^2 \delta^2 \gamma^2 \epsilon \zeta^2 + \beta^2 \delta^2 \gamma \epsilon \zeta^3}{(\beta + \gamma + \zeta)^4},
\\ \frac{s_{\eta_{1,5}^{[2]}}^{(5)}}{s_{\widetilde{\eta}_{1,5}^{[2]}}^{(5)}} & = (\beta + \gamma + \zeta)^4.
\end{align*}
Therefore, we can check (on a computer) for various small inputs ({\it e.g.}, $x_i^{(a)} \in \{1, \dotsc, 20\}$), we have
\[
J_1^{\{3\}} = \trop\left(\frac{s_{54443/4321}^{(0)}}{s_{54321/4321}^{(0)}}\right)(\bfX),
\qquad\qquad
J_2^{\{5\}} = \trop\left(s_{4321/321}^{(1)}\right)(\bfX),
\]
where $J_i^{\{t\}}$ is the rigging $J_i^{(1)}$ after $t$ time evolutions $T_1^{\infty}$.
\end{example}

\begin{remark}
Note that for $J_1^{(1)}$ in Example~\ref{ex:sl2} the disconnected box, which contributes a factor of $\beta + \gamma + \zeta$, cancels.
Thus we can safely omit disconnected components in $\eta_{a,t}^{[i]}$ (and $\widetilde{\eta}_{a,t}^{[i]}$).
\end{remark}

\begin{example}
\label{ex:sl3}
Consider type $\asl_3$ and $m = 3$. We have
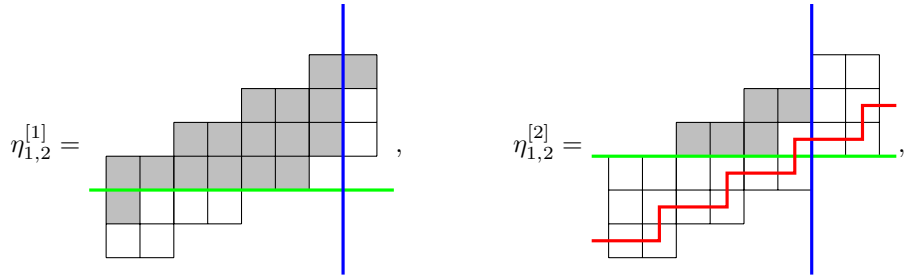
\[
\eta_{1,2}^{[1]} =
\begin{tikzpicture}[scale=.45, baseline=40]
\fill[color=lightgray] (0,1) rectangle (1,2);
\fill[color=lightgray] (0,2) rectangle (2,3);
\fill[color=lightgray] (2,2) rectangle (4,4);
\fill[color=lightgray] (4,2) rectangle (6,5);
\fill[color=lightgray] (6,3) rectangle (7,6);
\fill[color=lightgray] (7,5) rectangle (8,6);
\draw[step=1] (0,0) grid (2,3);
\draw[step=1] (2,1) grid (4,4);
\draw[step=1] (4,2) grid (6,5);
\draw[step=1] (6,3) grid (8,6);
\draw[green, very thick] (-.5, 2) -- (8.5, 2);
\draw[blue, very thick] (7, -.5) -- (7, 7.5);
\end{tikzpicture},
\qquad\qquad
\eta_{1,2}^{[2]} =
\begin{tikzpicture}[scale=.45, baseline=40]
\fill[color=lightgray] (2,3) rectangle (5,4);
\fill[color=lightgray] (4,4) rectangle (6,5);
\draw[step=1] (0,0) grid (2,3);
\draw[step=1] (2,1) grid (4,4);
\draw[step=1] (4,2) grid (6,5);
\draw[step=1] (6,3) grid (8,6);
\draw[green, very thick] (-.5, 3) -- (8.5, 3);
\draw[blue, very thick] (6, -.5) -- (6, 7.5);
\draw[red, very thick] (-0.5, 0.5) -- (1.5, 0.5) -- (1.5, 1.5) -- (3.5, 1.5) -- (3.5, 2.5) -- (5.5, 2.5) -- (5.5, 3.5) -- (7.5, 3.5) -- (7.5, 4.5) -- (8.5, 4.5);
\end{tikzpicture},
\]
and $\eta_{1,t}^{[3]} = \emptyset$ for all $t > 0$.
Therefore, we can check (on a computer) for various small inputs that
\[
J_1^{\{2\}} = \trop\left(\frac{s_{87761/642}^{(0)}}{s_{87531/642}^{(0)}}\right)(\bfX),
\qquad\qquad
J_2^{\{2\}} = \trop\left(s_{65/42}^{(1)}\right)(\bfX) = \trop\left(s_{43/2}^{(2)}\right)(\bfX),
\]
where $J_i^{\{t\}}$ is the rigging $J_i^{(1)}$ after $t$ time evolutions $T_1^{\infty}$.

Additionally, we have
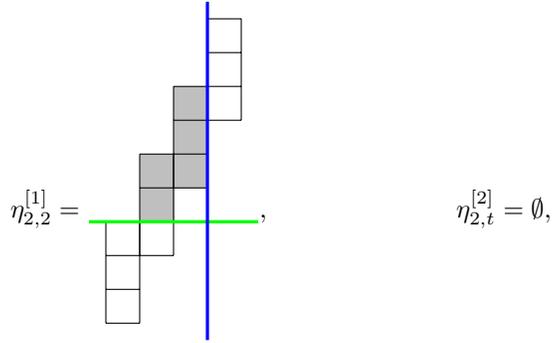
\[
\eta_{2,2}^{[1]} =
\begin{tikzpicture}[scale=.45, baseline=40]
\fill[color=lightgray] (1,3) rectangle (2,5);
\fill[color=lightgray] (2,4) rectangle (3,7);
\draw[step=1] (0,0) grid (1,3);
\draw[step=1] (1,2) grid (2,5);
\draw[step=1] (2,4) grid (3,7);
\draw[step=1] (3,6) grid (4,9);
\draw[green, very thick] (-.5, 3) -- (4.5, 3);
\draw[blue, very thick] (3, -.5) -- (3, 9.5);
\end{tikzpicture},
\hspace{70pt}
\eta_{2,t}^{[2]} = \emptyset,
\]
for all $t > 0$.
Therefore, we can check (on a computer) for various input that
\[
\widetilde{J}_1^{\{2\}} = \trop\left(s_{3332/2211}^{(2)}\right)(\bfX),
\]
where $\widetilde{J}_i^{\{t\}}$ is the rigging $J_i^{(2)}$ after $t$ time evolutions $T_2^{\infty}$.
\end{example}

\begin{remark}
\label{remark:sufficiently_large}
We note that because Conjecture~\ref{conj:formula_riggings} is \emph{eventually}, {\it i.e.} after sufficiently many time evolutions, the rigging is correct, any input $\bfX$ that is sufficiently large may not give the rigging under $\Phi$. Indeed, even for the natural condition that the rigged configuration being $a$-highest weight ($J_i^{(a)} \geq 0$ for all $i$) is not sufficient. Consider $m = 3$ for $\asl_3$ as in Example~\ref{ex:sl3}. Then for
\[
b = \boxed{1^{17}\, 2^8\, 3^{46}} \otimes \boxed{1^{48}\, 2^{42}\, 3^{36}} \otimes \boxed{1^{29}\, 2^{50}\, 3^{11}}\ ,
\]
we have $\trop(s^{(1)}_{65/42})(\bfX) = 143$, but $J_2^{\{2\}} = 144$. However, we have $\trop(s^{(0)}_{875/642})(\bfX) = 234 = J_2^{\{3\}}$.

Another natural approach for removing the necessity for time evolutions is to extend the number of copies of $\lambda(a,i)$ to construct $\eta_{a,t}^{[i]}$. However, the above shows that this would not remove this condition because $\lambda(1,2) = \emptyset$.
\end{remark}

\begin{example}
\label{ex:sl4}
Consider type $\asl_4$ and $m = 3$. We have
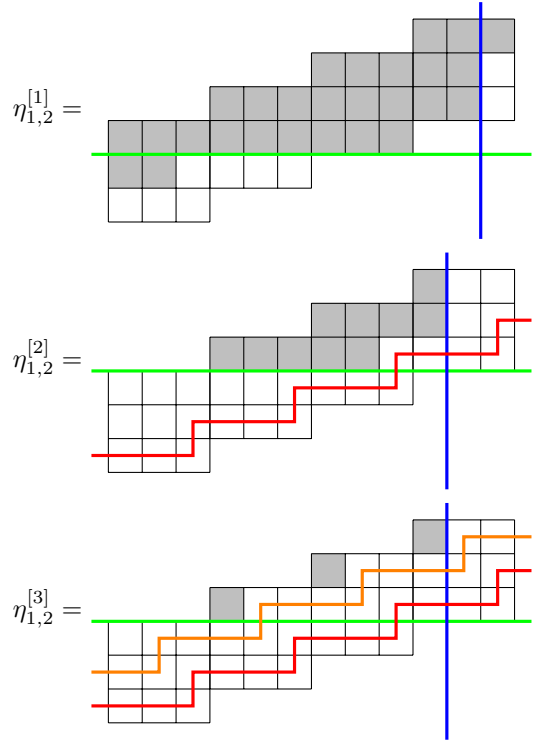
\begin{align*}
\eta_{1,2}^{[1]} & =
\begin{tikzpicture}[scale=.45, baseline=40]
\fill[color=lightgray] (0,1) rectangle (2,3);
\fill[color=lightgray] (0,2) rectangle (3,3);
\fill[color=lightgray] (3,2) rectangle (6,4);
\fill[color=lightgray] (6,2) rectangle (9,5);
\fill[color=lightgray] (9,3) rectangle (11,6);
\fill[color=lightgray] (11,5) rectangle (12,6);
\draw[step=1] (0,0) grid (3,3);
\draw[step=1] (3,1) grid (6,4);
\draw[step=1] (6,2) grid (9,5);
\draw[step=1] (9,3) grid (12,6);
\draw[green, very thick] (-.5, 2) -- (12.5, 2);
\draw[blue, very thick] (11, -.5) -- (11, 6.5);
\end{tikzpicture}
\\
\eta_{1,2}^{[2]} & =
\begin{tikzpicture}[scale=.45, baseline=40]
\fill[color=lightgray] (3,3) rectangle (8,4);
\fill[color=lightgray] (6,4) rectangle (10,5);
\fill[color=lightgray] (9,5) rectangle (10,6);
\draw[step=1] (0,0) grid (3,3);
\draw[step=1] (3,1) grid (6,4);
\draw[step=1] (6,2) grid (9,5);
\draw[step=1] (9,3) grid (12,6);
\draw[green, very thick] (-.5, 3) -- (12.5, 3);
\draw[blue, very thick] (10, -.5) -- (10, 6.5);
\draw[red, very thick] (-0.5, 0.5) -- (2.5, 0.5) -- (2.5, 1.5) -- (5.5, 1.5) -- (5.5, 2.5) -- (8.5, 2.5) -- (8.5, 3.5) -- (11.5, 3.5) -- (11.5, 4.5) -- (12.5, 4.5);
\end{tikzpicture}
\\
\eta_{1,2}^{[3]} & =
\begin{tikzpicture}[scale=.45, baseline=40]
\fill[color=lightgray] (3,3) rectangle (4,4);
\fill[color=lightgray] (6,4) rectangle (7,5);
\fill[color=lightgray] (9,5) rectangle (10,6);
\draw[step=1] (0,0) grid (3,3);
\draw[step=1] (3,1) grid (6,4);
\draw[step=1] (6,2) grid (9,5);
\draw[step=1] (9,3) grid (12,6);
\draw[green, very thick] (-.5, 3) -- (12.5, 3);
\draw[blue, very thick] (10, -.5) -- (10, 6.5);
\draw[red, very thick] (-0.5, 0.5) -- (2.5, 0.5) -- (2.5, 1.5) -- (5.5, 1.5) -- (5.5, 2.5) -- (8.5, 2.5) -- (8.5, 3.5) -- (11.5, 3.5) -- (11.5, 4.5) -- (12.5, 4.5);
\draw[orange, very thick] (-0.5, 1.5) -- (1.5, 1.5) -- (1.5, 2.5) -- (4.5, 2.5) -- (4.5, 3.5) -- (7.5, 3.5) -- (7.5, 4.5) -- (10.5, 4.5) -- (10.5, 5.5) -- (12.5, 5.5);
\end{tikzpicture}
\end{align*}
and $\eta_{1,2}^{[4]} = \emptyset$.
Therefore, we can check (on a computer and again for small values as per Remark~\ref{remark:sufficiently_large}) for various input that
\begin{gather*}
J_1^{\{2\}} = \trop\left(\frac{s_{12,11,11,10,2/963}^{(1)}}{s_{11,11,8,5/963}^{(1)}}\right)(\bfX),
\qquad\qquad
J_2^{\{2\}} = \trop\left(\frac{s_{775/63}^{(2)}}{s_{741/63}^{(2)}}\right)(\bfX),
\\
J_3^{\{2\}} = \trop\left(s_{\eta_{1,2}^{[3]}}^{(3)}\right)(\bfX) = \trop\left(s_{4/3}^{(3)}\right)(\bfX),
\end{gather*}
where $J_i^{\{t\}}$ is the rigging $J_i^{(1)}$ after $t$ time evolutions $T_1^{\infty}$.
\end{example}

\begin{prop}
\label{prop:single_factor}
Conjecture~\ref{conj:formula_shapes} and Conjecture~\ref{conj:formula_riggings} describe the image of $\Phi^{-1}$ when $m = 1$.
\end{prop}

\begin{proof}
Fix an element $b \in B^{1,s}$. Recall that $x_1^{(i)}$ equals the number of $i$ entries in $b$. We will show that after adding $k$ of the letters $i$ appearing in $b$, we have $\nu^{(a)}$ is a single row partition with
\[
\nu_1^{(a)} = \begin{cases}
k + x_1^{(i+1)} + \cdots + x_1^{(n)} & \text{if } a < i, \\
x_1^{(a)} + \cdots + x_1^{(n)} & \text{if } a > i.
\end{cases}
\]
Moreover, we have that each row is singular with
\[
J_1^{(a)} = \begin{cases}
-x_1^{(a+1)} & \text{if } a+1 > i, \\
-k & \text{if } a+1 = i, \\
0 & \text{if } a+1 < i.
\end{cases}
\]

Note that $\ls$ may only possibly change the riggings on $\nu^{(1)}$.
However, if $i \neq 1$, then $\ell = \nu_1^{(1)} = k + x_1^{(i+1)} + \cdots + x_1^{(n)}$, which equals the number of entries we have added from $b$, and so the vacancy number $p_{\ell}^{(1)}$ does not change under $\ls^{-1}$.
If $i = 1$, then $\delta^{-1}$ does nothing, and the current rigged configuration will equal $\Phi^{-1}(b)$.

Next, since the resulting rows are singular under $\delta^{-1}$, and each partition has exactly one (singular) row.
So if we are adding a box to $\nu^{(a)}$, which only occurs for $a < i$, it must be to $\nu_1^{(a)}$.
Thus our description above gives each step of $\Phi^{-1}(b)$.
By induction (the base case is trivially the empty rigged configuration and no letters added), the final rigged configuration is given by
\begin{align*}
\nu_1^{(a)} & = x_1^{(a+1)} + \cdots + x_1^{(n)},
\\ J_1^{(a)} & = -x_1^{(a+1)}.
\end{align*}

We note that because of the shift and that $a \geq 1$, the cylindric semistandard tableaux of shape $n - a$ is exactly equal to the semistandard tableaux of shape $n - a$ and all entries are $1$.
Note that $\lambda(a, 1) = \emptyset$, so the corresponding (cylindric) loop Schur function is equal to $1$.
Hence, we have
\[
\nu_1^{(a)} = \trop\left( \frac{\overline{s}_{\lambda(a,0),a }^{(0)}}{\overline{s}_{\emptyset,a}^{(0)}} \right)
= \trop\left(\overline{s}_{n-a,a}^{(0)}\right)
= x_1^{(a+1)} + \cdots + x_1^{(n)},
\]
and Conjecture~\ref{conj:formula_shapes} holds.
Furthermore, let $(\nu, \widetilde{J}) = T_a^{\infty}(\nu, J)$, and we have
\[
\widetilde{J}_1^{(a)} = x_1^{(a+2)} + \cdots + x_1^{(n)}
= \trop\left( \frac{s_{n-a-1}^{(0)}}{s_{\emptyset}^{(0)}} \right).
\]
Recall that we remove one row from the bottom when constructing $\eta_{a,1}^{[0]} = 2n - 2a - 1 / n - a$, but we would also shift by $n-a$.
Hence the loop Schur functions above agree with those given by Conjecture~\ref{conj:formula_riggings}, and so Conjecture~\ref{conj:formula_riggings} holds.
Note $J_1^{(a)} \leq 0$, and if $x_1^{(a)} > 0$, one time evolution is the first time the riggings $J^{(a)}$ are all non-negative.
When $x_1^{(a)} = 0$, all riggings are non-negative after zero time evolutions, but we have $J_1^{(a)} = 0 = \trop(s_{\emptyset}^{(0)})$.
\end{proof}

Note that in the proof of Proposition~\ref{prop:single_factor}, we see the necessity of applying sufficiently many time evolutions.


\section{Equating with the KSS bijection}
\label{sec:KSS_bijection}

\subsection{Fusing columns}

Recall that $\gamma$ and $\ls$ are the column splitting operations on rigged configurations and tableaux, respectively. We call the inverse process column fusing.

If we write the element in vector format, the operation of $\ls^{-1}$ is given by
\[
\column{0 \\ \vdots \\ 0 \\ 1 \\ 0 \\ \vdots \\ 0} \otimes
\column{0 \\ \vdots \\ 0 \\ x_1^{(k)} \\ x_1^{(k+1)} \\ \vdots \\ x_1^{(n)}} \otimes
\column{x_2^{(2)} \\ \vdots \\ x_2^{(k)} \\ x_2^{(k+1)} \\ x_2^{(k+2)} \\ \vdots \\ x_2^{(n+1)}}
\otimes \cdots \otimes
\column{x_m^{(m)} \\ \vdots \\ x_m^{(k+m-2)} \\ x_m^{(k+m-1)} \\ x_m^{(k+m)} \\ \vdots \\ x_1^{(n+m-1)}}
\longmapsto
\column{0 \\ \vdots \\ 0 \\ 0 \\ 0 \\ \vdots \\ 0} \otimes
\column{0 \\ \vdots \\ 0 \\ x_1^{(k)} + 1 \\ x_1^{(k+1)} \\ \vdots \\ x_1^{(n)}} \otimes
\column{x_2^{(2)} \\ \vdots \\ x_2^{(k)} \\ x_2^{(k+1)} \\ x_2^{(k+2)} \\ \vdots \\ x_2^{(n+1)}}
\otimes \cdots \otimes
\column{x_m^{(m)} \\ \vdots \\ x_m^{(k+m-2)} \\ x_m^{(k+m-1)} \\ x_m^{(k+m)} \\ \vdots \\ x_m^{(n+m-1)}}
\]

For the purposes of our proof, we will need to extend the definition of the KR crystals to $B^{1,0} \iso B(0)$, which is a $U_q'(\g)$-crystal with a single element $u_0$ such that
\[
\varepsilon_i(u_0) = \varphi_i(u_0) = 0,
\qquad\qquad
\wt(u_0) = 0,
\]
for all $i \in I$.

\begin{remark}
\label{remark:trivial_factor}
We note that the addition of a factor of $B^{1,0}$ on the left does not change the evaluation of the tropicalization of a (cylindric) loop Schur function. This is because adding a row with all entries $0$ on top of a (cylindric) semistandard tableau trivially satisfies the (cylindric) semistandard conditions and does not effect the weight since $x_0^{(a)} = 0$ for all $a \in I_0$.
\end{remark}

\begin{prop}
Let $b = \young(k) \otimes b_m \otimes \cdots \otimes b_1 \in B^{1,1} \otimes B^{1,s_m} \otimes \cdots \otimes B^{1,s_1}$ such that $x_{1}^{(a)} = 0$ for all $a < k$. Then we have
\begin{align*}
\trop\left(\frac{s^{(\ell)}_{\eta_{a,t}^{[i]}}}{s^{(\ell)}_{\widetilde{\eta}_{a,t}^{[i]}}}\right)(\bfX) & = \trop\left(\frac{s^{(\ell)}_{\eta_{a,t}^{[i]}}}{s^{(\ell)}_{\widetilde{\eta}_{a,t}^{[i]}}}\right)(\dot{\bfX}),
\\ \trop\left(\frac{\overline{s}^{(0)}_{\lambda(a, i-1),a}}{\overline{s}^{(0)}_{\lambda(a, i),a}}\right)(\bfX) & = \trop\left(\frac{\overline{s}^{(0)}_{\lambda(a, i-1),a}}{\overline{s}^{(0)}_{\lambda(a, i),a}}\right)(\dot{\bfX}) ,
\end{align*}
where $\dot{\bfX}$ are the variables of $\ls^{-1}(b)$ and we consider the leftmost factor to be $B^{1,0}$. Moreover, we have
\begin{align*}
\trop\left(\frac{s^{(\ell)}_{\eta_{a,t}^{[i]}}}{s^{(\ell)}_{\widetilde{\eta}_{a,t}^{[i]}}}\right)(\dot{\bfX}) & = J_i^{(a)},
\\ \trop\left(\frac{\overline{s}^{(0)}_{\lambda(a, i-1),a}}{\overline{s}^{(0)}_{\lambda(a, i),a}}\right)(\dot{\bfX}) & = \nu_i^{(a)},
\end{align*}
where $(\nu, J) = \gamma^{-1}\bigl(\Phi(b)\bigr)$,
if Conjecture~\ref{conj:formula_riggings} and Conjecture~\ref{conj:formula_shapes}, respectively, hold for $b$.
\end{prop}

\begin{proof}
From Remark~\ref{remark:trivial_factor}, the factor of $B^{1,0}$ on the left does not alter anything.
Consider a cylindric semistandard tableau $T$ of shape $\omega = (n-a)^{m+1}$ that realizes the minimum of $\trop(\overline{s}^{(0)}_{\omega})(\bfX)$.
It is sufficient to show that entries $0$ with color $k-1$ are in bijection with entries $1$ of color $k$.
Note that any entry with $0$ must appear in the first row of $T$ and those with $1$ on the first or second row.

Suppose $k = 1$. Note that there will never appear an entry $0$ with color $n$ as we color the first row as $(n-1, \dotsc, a + 1, a)$ and $n - (n-a) = a \geq 1$.
Similarly, there does not exist an entry $1$ with a color $1$ in the second row.
Now suppose there exists a $1$ with color $1$ in the first row.
We can replace the $1$ with color $1$ in the first row with a $0$, which will strictly decrease the weight.
This contradicts the weight minimality of $T$.
Hence there does not exist either a $1$ with color $1$ nor a $0$ with color $0$.

Suppose there exits a $1$ with color $k$ on the first row. We note that we can replace every entry strictly to the left with a $0$, and the resulting tableau will have at most the same weight as $T$. Thus we can replace the $1$ in the first row with a $0$ and weakly decrease the weight. Hence, we can replace the $1$ with a $0$ and not increase the weight.

Next, if a $1$ with color $k$ appears on the second row, then immediately above it, there must be an entry $0$ and it will have color $k-1$.
Now suppose there exists an entry $0$ in column $c$ with color $k-1$ in the first row.
If there does not exist an entry $1$ below in column $c$, then we can replace the entry $0$ with $1$.
However, this contradicts the weight minimality of $T$ as the resulting tableau will have strictly less weight.
Hence there must exists an entry of $1$ in column $c$ as desired.
\end{proof}

Thus from the recursive definition of $\Phi$, we can reduce Conjecture~\ref{conj:formula_riggings} and Conjecture~\ref{conj:formula_shapes} to showing the following.

\begin{conj}
Let $b = \young(k) \otimes b_m \otimes \cdots \otimes b_1 \in B^{1,1} \otimes B^{1,s_m} \otimes \cdots \otimes B^{1,s_1}$. Suppose
\begin{align*}
\trop\left(\frac{s^{(\ell)}_{\eta_{a,t}^{[i]}}}{s^{(\ell)}_{\widetilde{\eta}_{a,t}^{[i]}}}\right)(\widehat{\bfX}) & = \widehat{J}_i^{(a)},
\\ \trop\left(\frac{\overline{s}^{(0)}_{\lambda(a, i-1),a}}{\overline{s}^{(0)}_{\lambda(a, i),a}}\right)(\widehat{\bfX}) & = \widehat{\nu}_i^{(a)},
\end{align*}
where $(\widehat{\nu}, \widehat{J}) = \delta\bigl(\Phi(b)\bigr)$ and $\widehat{\bfX}$ are the variables of $\operatorname{lb}(b)$.
Then we have
\begin{align*}
\trop\left(\frac{s^{(\ell)}_{\eta_{a,t}^{[i]}}}{s^{(\ell)}_{\widetilde{\eta}_{a,t}^{[i]}}}\right)(\bfX) & = J_i^{(a)},
\\ \trop\left(\frac{\overline{s}^{(0)}_{\lambda(a, i-1),a}}{\overline{s}^{(0)}_{\lambda(a, i),a}}\right)(\bfX) & = \nu_i^{(a)},
\end{align*}
where $(\nu, J) = \Phi(b)$.
\end{conj}

\section*{Acknowledgements}

The author thanks Pavlo Pylyavskyy for numerous discussions, help in developing the conjectural formula presented in this work, and the proof of Lemma~\ref{lemma:single_period}.
The author also thanks Thomas Lam, Gabriel Frieden, and Reiho Sakamoto for comments on early drafts of this manuscript and discussions.
Additionally, the author thanks Anne Schilling for comments on an early draft of this manuscript.
Finally, the author thanks the anonymous referees for their useful comments.
This work benefited from computations using {\sc SageMath}~\cite{sage, combinat}.

\appendix

\section{{\sc SageMath} code}

We give our {\sc SageMath}~\cite{sage} code that was used to test Conjecture~\ref{conj:formula_riggings}.

We start by giving the functions to generate the tropicalized loop Schur functions $\trop(s_{\mu}^{(r)})$:
\begin{lstlisting}
def loop_schur(m, n, mu, r):
    SSYT = SemistandardSkewTableaux(mu, max_entry=m)
    vars = [var(','.join('x%s_%s'%(i+1,j+1) for i in range(n)))
            for j in range(m)]
    entry = lambda T, c: T[c[0]][c[1]]
    return min_symbolic(sum(vars[entry(T,c)-1][(c[0]-c[1]+r-1)%n]
                            for c in T.cells())
                        for T in SSYT)
\end{lstlisting}
the rigging $J_r^{(a+1)}$ (the $a+1$ is because {\sc SageMath} uses $0$-based indexing):
\begin{lstlisting}
def generate_rc_data(m, n, a, r, data, num=6):
    state = [[data[i+n*j] for i in range(n)] for j in range(m)]
    sizes = [[1, sum(st)] for st in state]
    construct = crystals.TensorProductOfKirillovReshetikhinTableaux
    KRT = construct(['A', n-1, 1], sizes)
    elt = KRT(pathlist=[sum(([i+1]*c for i,c in enumerate(st)), [])
                        for st in state])
    rc = elt.to_rigged_configuration()
    if not rc[a]:  # An empty partition
        return [0]*num
    return [rc[a].rigging[r] + i*rc[a][r] for i in range(num)]
\end{lstlisting}
and evaluating the tropical loop Schur function:
\begin{lstlisting}
def evaluate(f, m, n, data):
    d = {'x%s_%s'%((i+j)%n+1, j+1): data[i+n*j]
         for i in range(n) for j in range(m)}
    return f(**d)
\end{lstlisting}

Next we construct Example~\ref{ex:sl2} and verify Conjecture~\ref{conj:formula_riggings} on a data set:
\begin{lstlisting}
sage: numer = loop_schur(3,2, [[4,4,4,3], [3,2,1]], 0)
sage: denom = loop_schur(3,2, [[4,3,2,1], [3,2,1]], 0)
sage: data = [10, 4, 2, 4, 6, 2]
sage: evaluate(numer, 3, 2, data) - evaluate(denom, 3, 2, data)
30
sage: generate_rc_data(2, 3, 0, 0, data)
[-6, 6, 18, 30, 42, 54]
\end{lstlisting}

\bibliographystyle{alpha}
\bibliography{rc_shape}{}
\end{document}